\documentclass[11pt, a4paper]{amsart}
\usepackage{amssymb, array, amsmath, amscd, pdfpages, enumerate, amsthm,  setspace,mathtools}

\usepackage{graphics} 
\usepackage{mathptmx} 
\usepackage{graphicx,paralist}
\usepackage{float}
\usepackage{xcolor}
\usepackage{xfrac}
\usepackage{textcomp}
\usepackage{amsbsy}
\usepackage{enumitem}
\usepackage{cite}

\newtheorem{theorem}{Theorem}[section]
\newtheorem{lemma}[theorem]{Lemma}
\theoremstyle{definition}
\newtheorem{definition}[theorem]{Definition}
\newtheorem{assumption}[theorem]{Assumption}
\newtheorem{remark}[theorem]{Remark}

\begin{document}

\title{Decentralized Observer Design for Virtual Decomposition Control}

\author{Jukka-Pekka Humaloja}
\address{Tampere University, Faculty of Information Technology and Communication Sciences, Mathematics, P.O. Box 692, 33014 Tampere University,  Finland}
\email{jukka-pekka.humaloja@tuni.fi}
\email{lassi.paunonen@tuni.fi}
\author{Janne Koivum\"aki}
\address{Tampere University, Faculty of Engineering and Natural Sciences, Automation
  Technology and Mechanical Engineering, P.O. Box 589, 33014 Tampere University, Finland}
  \email{janne.koivumaki@tuni.fi}
  \email{jouni.mattila@tuni.fi}
  \author{Lassi Paunonen}
  \author{Jouni Mattila}

\thispagestyle{plain}

\begin{abstract}
In this paper, we incorporate velocity observer design into the virtual decomposition control (VDC) 
strategy of an $n$-DoF open chain robotic manipulator. Descending from the VDC strategy, the 
proposed design is based on decomposing the $n$-DoF manipulator into subsystems, i.e., rigid links 
and joints, for which the decentralized controller-observer implementation can be done locally. 
Similar to VDC, the combined controller-observer design is passivity-based, and we show that it 
achieves semiglobal exponential convergence of the tracking error. The convergence analysis is 
carried out using Lyapunov functions based on the observer and controller error dynamics. The 
proposed design is demonstrated in a simulation study of a 2-DoF open chain robotic manipulator in 
the vertical plane.
\end{abstract}



\thanks{J.-P. Humaloja and L. Paunonen are supported by the Academy of Finland Grant number 310489 held by L. Paunonen}
\thanks{L. Paunonen has been funded by the Academy of Finland Grant number 298182.}
\thanks{J. Mattila has been funded by the Academy of Finland Grant number 283171}

\maketitle

\section{Introduction}

The virtual decomposition control (VDC) approach \cite{ZhuBook, Zhu1997} is a nonlinear 
model-based control method that is developed for controlling complex systems, and it has been 
demonstrated to be very effective especially in robotic control 
\cite{Zhu1998adaptive,Zhu2000,Zhu2013,Koivumaki_TRO2015,
	Koivumaki_CEP2019}. The fundamental idea of VDC is that the system can be virtually 
	decomposed 
into \textit{modular subsystems} (such as rigid links and joints), allowing a decentralized control that 
can be designed locally at the subsystem level and that guarantees stability by fully taking into 
account the dynamic interactions among adjacent subsystems. The VDC methodology is introduced 
in greater detail in Section \ref{sec:VDC}.

The existing VDC literature requires that the position and velocity states of the system are 
measurable for the control design. While position measurements can be done accurately, the 
instruments for measuring rotation speed, e.g., tachometers, are known to be often contaminated 
with noise. Velocity data can naturally be  obtained by numerical differentiation of the position 
sensor data but there is no theoretical justification for this method \cite{BerNij93, Ber18arxiv}. Due 
to these challenges, control of $n$-DoF robotic manipulators without velocity data has been 
extensively studied, e.g., in \cite{BerNij93, NicTom90, ZhuChe92, BurDav97, ZerDix99, MalDri12, 
Dri15}, see also the survey \cite{Ber18arxiv}, where the actuator dynamics have been neglected. 
Similar research for robots with, e.g., hydraulic actuators has been done in \cite{BuFYaoICRA00, 
SirSal01}. Our approach to the proposed controller-observer design is inspired by the 
passivity-based design in \cite{BerNij93} and the subsystem-based VDC approach \cite{ZhuBook}.

In this paper, we design a control law for an $n$-DoF open chain robotic manipulator (see  Fig. 
\ref{fig:system}) in such a way that the position trajectories $q_i(t)$ of the $n$ joints follow given 
\textit{desired trajectories} $q_{id}(t)$. Velocity data is not available for the control design, due to 
which we design a velocity observer based on position (measurement) and torque (input) data. We 
note that the manipulator in Fig. \ref{fig:system} is in a planar joint configuration for the sake of 
graphical simplicity; the system kinematics and dynamics are provided in the general 6-DoF 
matrix/vector form instead of the scalar presentation and the joint orientations may be arbitrary.

The main contribution of the paper is incorporating an observer design into the VDC methodology, 
which yields a novel decentralized controller-observer design for robotic manipulators. In 
comparison to the existing literature \cite{BerNij93, Ber18arxiv, NicTom90, ZhuChe92, BurDav97, 
	ZerDix99, MalDri12, Dri15,BuFYaoICRA00, SirSal01} where the designs are based on the dynamics 
	of 
the whole manipulator, in the proposed decentralized design the control and observer gains are 
proportional to the individual link/joint dynamics. Thus, in the proposed design the gain conditions 
remain unaltered even if the complexity of the system (number of DoFs) increases, which is not the 
case for the existing designs where the gain conditions depend on the whole system dynamics. 
Moreover, the proposed design is highly modular in the sense that if parts were replaced in or added 
to the manipulator, the controller-observer design needs to be reimplemented only for the new parts 
while the other parts remain intact. Our main result, semiglobal exponential convergence of the 
proposed design, is presented in Theorem \ref{thm:stability} in Section \ref{sec:stability}. Thereafter, 
Remark \ref{rem:appl} discusses possible extensions of the design and addresses arbitrary joint 
configurations for the $n$-DoF manipulator. Semiglobality of the achieved convergence originates 
from the requirement of the link velocities being bounded, albeit they may be arbitrarily large.  We 
note that a similar approach has been taken, e.g., in \cite{BerNij93}.

\begin{figure*}
	\includegraphics[width=1.00\textwidth]{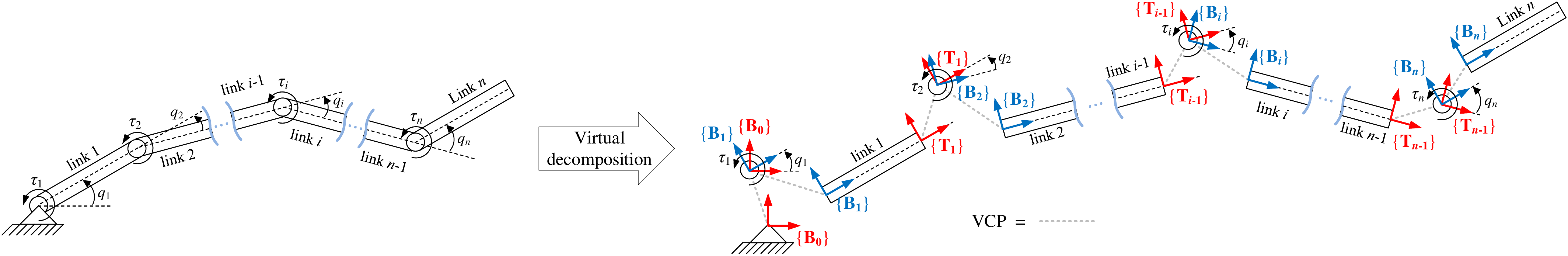}
	\caption{The $n$-DoF open chain robotic manipulator and its virtual decomposition.}
	\label{fig:system}
\end{figure*}

The proposed controller-observer design is based on the VDC design principles in the sense that the 
controller and observer are designed for the links and joints, i.e., the virtual subsystems, individually, 
and the stability analysis of the error dynamics can be carried out locally at the subsystem level in 
terms of \textit{virtual stability} (see Section \ref{sec:VDC}). The idea is to construct Lyapunov 
functions for the subsystems based on the error dynamics, using which the error dynamics can be 
shown to be exponentially stable. The approach is an integral part of VDC in controller stability 
analysis, and in this paper we extend the design and analysis to account for the observer error 
dynamics as well. It should be noted that due to the nonlinear dynamics of the system, the 
separation principle cannot be utilized in stability analysis but the controller and observer 
convergences must be shown simultaneously.

The paper is organized as follows. In Section \ref{sec:math}, we present preliminaries concerning link 
dynamics, stability analysis and the VDC methodology. In Section \ref{sec:system}, we present the 
kinematics and dynamics of the system model. In Section \ref{sec:obs}, we present the decentralized 
joint and link velocity observer designs which are incorporated into the VDC design in Section 
\ref{sec:ctrl}. Exponential stability of the controller and observer error dynamics is shown in Section 
\ref{sec:stability}. In Section \ref{sec:sim}, the proposed design is demonstrated by a numerical 
simulation on a 2-DoF robot in the vertical plane. Finally, the paper is concluded in Section 
\ref{sec:concl}.

\section{Mathematical Preliminaries} \label{sec:math}

\subsection{Dynamics of a Rigid Body} \label{mp:dyn}

Consider an orthogonal, three-dimensional coordinate system $\left\{
\mathbf{A} \right\}$ (frame $\left\{ \mathbf{A} \right\}$) attached to
a rigid body. Let $^{\mathbf{A}} \mathbf{v} \in
\mathbb{R}^3$ and $^{\mathbf{A}} \boldsymbol{\omega}\in \mathbb{R}^3$ be
the linear and angular velocity vectors, respectively, of frame
$\left\{ \mathbf{A} \right\}$, expressed in frame $\left\{ \mathbf{A}
\right\}$ (see \cite[Sect. 2.5]{ZhuBook} for expressing velocities and forces in body frames). To 
facilitate the transformations of velocities among
different frames, the linear/angular velocity vector of frame
\{\textbf{A}\} can be written as
\begin{equation} \label{eq:vel}
	{}^{\mathbf A}V := \begin{bmatrix}
		^{\mathbf A}{{\mathbf
				v}} \\ {}^{{\mathbf A}}{{\mathbf \omega}}
	\end{bmatrix} \in {{\mathbb R}}^6.
\end{equation}
In a similar manner, let ${}^\mathbf{A}\mathbf{p} \in \mathbb{R}^3$
and ${}^{\mathbf{A}}\boldsymbol{\varphi} \in \mathbb{R}^3$ be the linear and angular
position vectors, respectively, of frame $\left\{ \mathbf{A} \right\}$
and define
$^{\mathbf{A}}P := \begin{bmatrix}
	^\mathbf{A} \mathbf{p} \\ ^{\mathbf{A}}\boldsymbol{\varphi}
\end{bmatrix} \in \mathbb{R}^6$,
so that $\displaystyle \frac{d}{dt}(^{\mathbf{A}}P) = {}^{\mathbf{A}}V$.

Let $^{\mathbf{A}} \mathbf{f} \in \mathbb{R}^3$ and
$^{\mathbf{A}} \mathbf{m} \in \mathbb{R}^3$ be the force and
moment vectors applied to the origin of frame $\left\{ \mathbf{A} \right\}$,
expressed in frame $\left\{ \mathbf{A} \right\}$. Similar to \eqref{eq:vel}, the
force/moment vector in frame $\left\{ \mathbf{A} \right\}$ can be written as
\begin{equation} \label{eq:f}
	{}^{{\mathbf A}}F := \begin{bmatrix}
		{}^{{\mathbf A}}{{\mathbf f}} \\ ^{{\mathbf A}}{{\mathbf m}}
	\end{bmatrix} \in {{\mathbb R}}^6.
\end{equation}

Consider two given frames, denoted as $\{ \mathbf{A}\}$ and
\{\textbf{B}\}, fixed to a common rigid body. The following relations
hold: 
\begin{subequations}
	\label{eq:frametrans}%
	\begin{align}
		{}^{{\mathbf B}}V &= {}^{{\mathbf A}}{{{\mathbf U}}^T_{{\mathbf
					B}}}{}^{{\mathbf A}}V \\
		{}^{{\mathbf A}}F &= {}^{{\mathbf A}}{{{\mathbf U}}_{{\mathbf B}}}{}^{{\mathbf B}}F,
	\end{align}
\end{subequations}
where $^{\mathbf{A}} \mathbf{U}_{\mathbf{B}}\in\mathbb{R}^{6\times6}$ denotes a
force/moment transformation matrix that transforms the force/moment
vector measured and expressed in frame $\left\{ \mathbf{B} \right\}$ to the same
force/ moment vector measured and expressed in frame $\left\{ \mathbf{A} \right\}$ (see \cite[Sect. 
2.5.3]{ZhuBook} for details).

Let frame $\left\{ \mathbf{A} \right\}$ be fixed to a rigid body. The rigid body
dynamics, expressed in frame $\left\{ \mathbf{A} \right\}$, can be written as
\begin{equation} \label{eq:linkdyn}
	{{\mathbf M}}_{{\mathbf
			A}}\frac{d}{dt}({}^{{\mathbf A}}V)+{{\mathbf C}}_{{\mathbf
			A}}({}^{{\mathbf A}}{\omega }){}^{{\mathbf A}}V+{{\mathbf
			G}}_{{\mathbf A}}={}^{{\mathbf A}}F^*
\end{equation}
where ${}^{{\mathbf A}}F^{
	*}\in\mathbb{R}^{6}$ is the net force/moment vector of the rigid body
expressed in frame $\left\{ \mathbf{A} \right\}$ and
$\mathbf{M}_{\mathbf{A}}\in\mathbb{R}^{6\times 6}$,
$\mathbf{C}_{\mathbf{A}}(^{\mathbf{A}} \omega) \in
\mathbb{R}^{6\times6}$ and $\mathbf{G}_{\mathbf{A}}\in\mathbb{R}^{6}$
are the mass matrix, the Coriolis/centrifugal matrix and the
gravity vector, respectively (see \cite[Sect. 2.6.2]{ZhuBook} for the detailed expressions).
We note that by the structure of
matrix $\mathbf{C}_{\mathbf{A}}(\cdot)$, it has the following properties
\begin{subequations}
	\label{eq:Cprop}% 
	\begin{align}
		\mathbf{C}_{\mathbf{A}}(^{\mathbf{A}}\omega_1)^T & = - 
		\mathbf{C}_{\mathbf{A}}(^{\mathbf{A}}\omega_1) \label{eq:Cp1} \\ 
		\mathbf{C}_{\mathbf{A}}(\alpha_1{}^{\mathbf{A}}\omega_1 
		+\alpha_2{}^{\mathbf{A}}\omega_2)  & = 
		\alpha_1\mathbf{C}_{\mathbf{A}}(^{\mathbf{A}}\omega_1) + 
		\alpha_2\mathbf{C}_{\mathbf{A}}(^{\mathbf{A}}\omega_2) \label{eq:Cp2} \\ 
		\| \mathbf{C}_{\mathbf{A}}(^{\mathbf{A}}\omega_1)\| & \leq 
		M_{c,\mathbf{A}}\|^{\mathbf{A}}\omega_1\| \label{eq:Cp3}
	\end{align}
\end{subequations}
for some $M_{c,\mathbf{A}} > 0$ and for all $\alpha_1,\alpha_2>0$ and $^{\mathbf{A}}\omega_1, 
{}^{\mathbf{A}}\omega_2 \in \mathbb{R}^3$.

\subsection{Virtual Decomposition Control} \label{sec:VDC}

Virtual decomposition control (VDC) is a control design method where the original system is
decomposed into subsystems by placing conceptual \textit{virtual
	cutting points} (VCP) \cite[Def. 2.13]{ZhuBook}. Every such cutting point forms a virtual
cutting surface on the rigid body, where three-dimensional force vectors and
three-dimensional moment vectors can be exerted from one part to
another. Fig. \ref{fig:system} displays a virtual decomposition of an
$n$-DoF robot and the virtual cutting points.

Adjacent subsystems resulting from a virtual decomposition have
dynamic interactions with each other. These interactions are uniquely
defined by scalar terms called \textit{virtual power flows} (VPFs)
\cite[Def. 2.16]{ZhuBook}. With respect to frame $\{{\mathbf A}\}$,
the virtual power flow is given by 
\begin{equation} p_{\mathbf{A}} = ({}^{{\mathbf A}}V_{
		r}-{}^{{\mathbf A}}V)^{T}({}^{{\mathbf A}}F_{r}-{}^{{\mathbf
			A}}F) \label{eq:VPF}
\end{equation} where ${}^{{\mathbf A}}V_{r} \in \mathbb{R}^{6}$
and ${}^{{\mathbf A}}F_{r} \in \mathbb{R}^{6}$ represent the
required vectors of ${}^{{\mathbf A}}V \in \mathbb{R}^{6}$ and
${}^{{\mathbf A}}F \in \mathbb{R}^{6}$, respectively, that will be
presented in Section \ref{sec:ctrl}. 

The VPFs are closely related to \textit{virtual stability}
\cite[Def. 2.17]{ZhuBook} which is the key concept of VDC. Virtual
stability is a tool for analyzing the stability of the system on a
subsystem level, where the subsystems do not need to be 
stable but using VPFs to represent dynamic interactions among adjacent subsystems. Motivated by
\cite[Def. 2.17]{ZhuBook}, we define virtual stability as follows:
\begin{definition}
	\label{def:vs}
	Consider a subsystem $i$ with dynamics $\dot{x}_i = g_i(t,x_i)$
	where $g_i$ is piecewise continuous in $t$ and locally Lipschitz in
	$x_i$. The subsystem $i$ is called \textit{virtually stable} if
	there exists a continuously differentiable function $\nu_i(t,x_i)$ such
	that 
	\begin{subequations}
		\label{eq:vslyap}% 
		\begin{align}
			\alpha_{i,1}\|x_i\|^2 & \leq \nu_i(t,x_i) \leq \alpha_{i,2}\|x_i\|^2 \\
			\dot{\nu}_i(t, x_i) & \leq -\alpha_{i,3}\|x_i\|^2 + p(t,x_i)_{\mathbf{A}_{i+1}} -
			p(t,x_i)_{\mathbf{A}_{i-1}} \label{eq:vslyap2}
		\end{align}
	\end{subequations}
	for all $t \geq 0$ and some $\alpha_{i,1},\alpha_{i,2},\alpha_{i,3} > 0$, where
	$p_{\mathbf{A}_{i+1}}$ and $p_{\mathbf{A}_{i-1}}$ are VPFs with respect to
	frames $\{ \mathbf{A}_{i+1}\}$ and $\{ \mathbf{A}_{i-1}\}$,
	respectively, adjacent to subsystem $i$.
\end{definition}

Note that without the VPFs, Definition \ref{def:vs} would coincide
with Lyapunov criteria for exponential stability as in
\cite[Thm. 4.10]{KhalBook}. The virtually stable subsystems will in
fact result in exponential stability of the entire system, as the VPFs in
\eqref{eq:vslyap2} will cancel out when we take the sum of the functions
$\nu_i$ over all the subsystems. We will prove the stability of the entire system in Theorem 
\ref{thm:stability}. The approach is the same as in
\cite[Thm. 2.1]{ZhuBook}, even though the stability arguments here are different.

\begin{remark}
	The VDC design is decentralized and local in the sense that changing
	the control (or dynamics) of a subsystem does not affect the control
	equations of the rest of the system as long as the VPFs among adjacent
	subsystems cancel out. We also note that the general concept of
	virtual stability in \cite[Def. 2.17]{ZhuBook} allows several VPFs
	between the subsystems. That is, the concept is not restricted to open
	chain systems but the restriction is made here merely for
	simplicity. The general formulation of VDC is given in
	\cite[Sect. 4]{ZhuBook}.
\end{remark}

\section{The System Model} \label{sec:system}

Consider the robot with $n$ links as in Fig.
\ref{fig:system} with the given virtual decomposition. For the sake of
generality, we will formulate 
the kinematics and dynamics of the system in the matrix/vector
form in $\mathbb{R}^6$. For more detailed consideration
of the kinematics and dynamics, see \cite[Chap. 3]{ZhuBook} where the
consideration is done for a 2-DoF robot.

\subsection{Kinematics}
\label{sec:kinematics}

Let the system base frame $\left\{ \mathbf{B}_0 \right\}$ have zero
velocity, i.e., $^{\mathbf{B}_0}V  = \mathbf{0}$. Then, using the
notation of Section \ref{mp:dyn}, the kinematics of an arbitrary joint
$i$ can be written as
\begin{equation}
	{}^{\mathbf{B}_i}V = \mathbf{z}_\tau\dot{q}_i +
	{}^{\mathbf{B}_{i-1}}{\mathbf{U}}^T_{{\mathbf{B}_i}}{}^{\mathbf{B}_{i-1}}V,
	\quad i\in \left\{ 1,2,\ldots,n \right\}, \label{eq:B1V}
\end{equation}
where $\mathbf{z}_\tau = [0\ 0\ 0\ 0\ 0\ 1]^T$, $\dot{q}_i$ is the
angular velocity of joint $i$ and ${}^{\mathbf{B}_{i-1}}{\mathbf{U}}_{{\mathbf{B}_i}}$ is as in 
\eqref{eq:frametrans}. Moreover, as the following relation holds for transforming the velocity vectors 
in link $i$:
\begin{equation}
	{}^{\mathbf{T}_i}V =
	{}^{\mathbf{B}_i}{\mathbf{U}}^T_{{\mathbf{T}_i}}{}^{\mathbf{B}_i}V,
	\quad i \in \left\{ 1,2,\ldots, n \right\}, \label{eq:T1V}
\end{equation}
the kinematics of joint $i$ can alternatively be written based
on the link velocities as 
\begin{equation}
	\label{eq:B1Valt}
	^{\mathbf{B}_i}V = \mathbf{z}_{\tau}\dot{q}_i +
	{}^{\mathbf{T}_{i-1}}U_{\mathbf{B}_i}^T{}^{\mathbf{T}_{i-1}}V, \quad i \in
	\left\{ 2,3,\ldots, n \right\}.
\end{equation}

\subsection{Single Link Dynamics in Cartesian Space}

As in \eqref{eq:linkdyn}, the motion dynamics of an arbitrary rigid link $i$ is expressed in frame
\{$\mathbf{B}_i$\} by 
\begin{equation}
	{{\mathbf M}}_{\mathbf{B}_i}\frac{d}{dt}({}^{\mathbf{B}_i}V) +
	{{\mathbf C}}_{\mathbf{B}_i}({}^{\mathbf{B}_i}{\omega
	}){}^{\mathbf{B}_i}V + {{\mathbf G}}_{\mathbf{B}_i} =
	{}^{\mathbf{B}_i}F^*, \quad i \in \left\{ 1,2,\ldots,n \right\}. \label{eq:B1F*}
\end{equation}
Furthermore, the resultant forces/moments of
link $i$ can be expressed as
\begin{equation}
	{}^{\mathbf{B}_i}F = {}^{\mathbf{B}_i}F^* +
	{}^{\mathbf{B}_i}{\mathbf{U}}_{{\mathbf{T}_i}}{}^{\mathbf{T}_i}F,
	\quad i \in \left\{ 1,2,\ldots,n \right\}, \label{eq:B1F}
\end{equation}
where $^{\mathbf{T}_n}F = \mathbf{0}$ as we assume that no external force/moment is imposed on
the origin of the frame $\left\{ \mathbf{T}_n \right\}$. Moreover, the
force/moment vector in frame $\left\{ \mathbf{B}_0 \right\}$ can be
written as 
\begin{equation}
	\label{eq:B0F}
	^{\mathbf{B}_0}F = {}^{\mathbf{B}_0}U_{B_1}{}^{\mathbf{B}_1}F.
\end{equation}

\subsection{Single Joint Dynamics in Joint Space}

The actuation torque of an arbitrary joint $i$
can be obtained from \eqref{eq:B1F} as
\begin{equation}
	\tau_{ai} =
	\mathbf{z}^T_\tau{}^{\mathbf{B}_i}F, \quad i \in \left\{
	1,2,\ldots,n \right\}. \label{eq:tau_a1} 
\end{equation}
Then, similarly to \cite[(3.51)]{ZhuBook}, joint $i$ torque $\tau_i$ (torque
input) can be written as
\begin{equation}
	\tau_i = I_{m,i}\ddot{q}_i + f_{c,i}(\dot{q}_i) + \tau_{ai},
	\quad i \in \left\{ 1,2,\ldots,n \right\} \label{eq:tau_1} 
\end{equation}
where $I_{m,i}$ is the joint moment of inertia and $f_{c,i}$ is
a Coulomb friction function model. The  friction model is
assumed to be increasing, globally Lipschitz continuous
and antisymmetric, e.g., Coulomb-viscous model
(see \cite[Sect. 2.3]{AndSod07}) Note that by monotonicity, the friction function model satisfies 
\begin{equation}
	\label{eq:frprop1} 
	-(x_1-x_2)(f_{c,i}(x_1)-f_{c,i}(x_2)) \leq 0.
\end{equation}
We note that the monotonicity assumption could be lifted if the  first
time derivatives of the functions $f_{c,i}$ are bounded, so that more
advanced friction models (see \cite[Sect. 3]{AndSod07}) could be
incorporated as well.

\section{Observer Design} \label{sec:obs}

In this section, we will consider velocity observers for arbitrary
link $i$ and joint $i$ motivated by the passivity-based observer design of
\cite[Sect. II.B]{BerNij93}. For the design, we need to have position
and torque data available. The final observer designs
must be done simultaneously with the control designs (see
\cite[Sect. II.C]{BerNij93}), which we will do in Section \ref{sec:ctrl}, where the following auxiliary 
analysis will be utilized. We make the following standing assumption:

\begin{assumption}
	Joint torque data $\tau_i$ and position data $q_i$ are available for the observer design for all $i 
	\in \{1,2,\ldots,n\}$.
\end{assumption}

\subsection{Observer for Link $i$} \label{obs:link}

Consider an observer system of the form
\begin{subequations}
	\label{eq:obsl1}% 
	\begin{align}
		^{\mathbf{B}_i}\dot{\hat{P}} & = ^{\mathbf{B}_i}Z -
		\mathbf{M}_{\mathbf{B}_i}^{-1} 
		\mathbf{L}_{\mathbf{B}_i}(^{\mathbf{B}_i}\hat{P}-{}^{\mathbf{B}_i}P)
		\\
		\mathbf{M}_{\mathbf{B}_i}{}^{\mathbf{B}_i}\dot{Z}
		& = ^{\mathbf{B}_i}F^{*} -
		\mathbf{C}_{\mathbf{B}_i}(^{\mathbf{B}_i}\hat{\boldsymbol{\omega}})
		^{\mathbf{B}_i}\hat{V} - \mathbf{G}_{\mathbf{B}_i}  \label{eq:obsl12}
	\end{align}
\end{subequations}
where $\mathbf{L}_{\mathbf{B}_i} > 0$ is an error
feedback gain matrix, $[^{\mathbf{B}_i}\hat{P} \quad
^{\mathbf{B}_i}Z]^T$ is the
observer state and $^{\mathbf{B}_i}\hat{V} =
{}^{\mathbf{B}_i}\dot{\hat{P}}$ is the observed velocity. Note that
the second line of the observer simply copies the link dynamics \eqref{eq:B1F*}.

Subtracting \eqref{eq:B1F*} from \eqref{eq:obsl12},
we obtain error dynamics
\begin{equation}
\mathbf{M}_{\mathbf{B}_i}(^{\mathbf{B}_i}\dot{\hat{V}} -
{}^{\mathbf{B}_i}\dot{V}) =
\mathbf{C}_{\mathbf{B}_i}(^{\mathbf{B}_i}\boldsymbol{\omega})^{\mathbf{B}_i}V -
\mathbf{C}_{\mathbf{B}_i}(^{\mathbf{B}_i}\hat{\boldsymbol{\omega}})^{\mathbf{B}_i}\hat{V} 
- \mathbf{L}_{\mathbf{B}_i}(^{\mathbf{B}_i}\hat{V} - {}^{\mathbf{B}_i}V).
\end{equation}
If we define a quadratic function $ \nu_{\mathbf{B}_i, obs}$ as 
\begin{equation}
	\label{eq:vb1obs}
	\nu_{\mathbf{B}_i, obs} := \frac{1}{2}(^{\mathbf{B}_i}\hat{V} -
	{}^{\mathbf{B}_i}V)^T \mathbf{M}_{\mathbf{B}_i}(^{\mathbf{B}_i}\hat{V} -
	{}^{\mathbf{B}_i}V),
\end{equation}
then
\begin{equation}
	\label{eq:vobsB11}
	\begin{split}
		\dot{\nu}_{\mathbf{B}_i,obs} & = (^{\mathbf{B}_i}\hat{V} -
		{}^{\mathbf{B}_i}V)^T[\mathbf{C}_{\mathbf{B}_i}(^{\mathbf{B}_i}\boldsymbol{\omega})^{\mathbf{B}_i}V
		 -
		\mathbf{C}_{\mathbf{B}_i}(^{\mathbf{B}_i}\hat{\boldsymbol{\omega}})^{\mathbf{B}_i}\hat{V}
		] \\
		& \quad - (^{\mathbf{B}_i}\hat{V} - {}^{\mathbf{B}_i}V)^T
		\mathbf{L}_{\mathbf{B}_i} (^{\mathbf{B}_i}\hat{V} -
		{}^{\mathbf{B}_i}V).
	\end{split}
\end{equation}
The term associated with the Coriolis/centrifugal forces can be
written as
\begin{align}
	& \mathbf{C}_{\mathbf{B}_i}(^{\mathbf{B}_i}\boldsymbol{\omega})^{\mathbf{B}_i}V -
	\mathbf{C}_{\mathbf{B}_i}(^{\mathbf{B}_i}\hat{\boldsymbol{\omega}})^{\mathbf{B}_i}\hat{V} \\
	& = \mathbf{C}_{\mathbf{B}_i}(^{\mathbf{B}_i}\boldsymbol{\omega})^{\mathbf{B}_i}V -
	\mathbf{C}_{\mathbf{B}_i}(^{\mathbf{B}_i}\hat{\boldsymbol{\omega}})^{\mathbf{B}_i}V
	- \mathbf{C}_{\mathbf{B}_i}(^{\mathbf{B}_i}\hat{\boldsymbol{\omega}})(^{\mathbf{B}_i}\hat{V}
	- {}^{\mathbf{B}_i}V), \nonumber
\end{align}
so by using \eqref{eq:Cp1}--\eqref{eq:Cp2} we obtain
\begin{equation}
	\label{eq:Cprops}
	\begin{split}
		& (^{\mathbf{B}_i}\hat{V} -
		{}^{\mathbf{B}_i}V)^T[\mathbf{C}_{\mathbf{B}_i}(^{\mathbf{B}_i}\boldsymbol{\omega})^{\mathbf{B}_i}V
		 -
		\mathbf{C}_{\mathbf{B}_i}(^{\mathbf{B}_i}\hat{\boldsymbol{\omega}})^{\mathbf{B}_i}\hat{V}
		] \\
		& = (^{\mathbf{B}_i}\hat{V} -
		{}^{\mathbf{B}_i}V)^T[\mathbf{C}_{\mathbf{B}_i}(^{\mathbf{B}_i}\boldsymbol{\omega})^{\mathbf{B}_i}V
		-
		\mathbf{C}_{\mathbf{B}_i}(^{\mathbf{B}_i}\hat{\boldsymbol{\omega}})^{\mathbf{B}_i}V]
		\\ & =  (^{\mathbf{B}_i}\hat{V} -
		{}^{\mathbf{B}_i}V)^T\mathbf{C}_{\mathbf{B}_i}(^{\mathbf{B}_i}\boldsymbol{\omega}
		-{}^{\mathbf{B}_i}\hat{\boldsymbol{\omega}})^{\mathbf{B}_i}V.
	\end{split}
\end{equation}

Let us now assume that the velocity vector $^{\mathbf{B}_i}V$ is
bounded, i.e., $\sup_{t > 0}\|^{\mathbf{B}_i}V\| =
M_{v,i} < \infty$. Continuing from \eqref{eq:Cprops} and using the
relative boundedness \eqref{eq:Cp3} of $\mathbf{C}_{\mathbf{B}_i}(\cdot)$, we obtain 
\begin{align}
	\|(^{\mathbf{B}_i}\hat{V} -
	{}^{\mathbf{B}_i}V)^T\mathbf{C}_{\mathbf{B}_i}(^{\mathbf{B}_i}\boldsymbol{\omega}
	-{}^{\mathbf{B}_i}\hat{\boldsymbol{\omega}})^{\mathbf{B}_i}V\|
	\leq
	\|^{\mathbf{B}_i}\hat{V}-{}^{\mathbf{B}_i}V
	\|M_{c,i}M_{v,i}\|^{\mathbf{B}_i}\hat{V}-{}^{\mathbf{B}_i}V\|. 
\end{align}
Utilizing the preceding identities and estimates in \eqref{eq:vobsB11}, we finally obtain 
\begin{equation}
	\label{eq:obsB1f}
	\dot{\nu}_{\mathbf{B}_i,obs} \leq -(^{\mathbf{B}_i}\hat{V} -
	{}^{\mathbf{B}_i}V)^T(\mathbf{L}_{\mathbf{B}_i} - M_{c,i}M_{v,i}I_{6\times 
	6})(^{\mathbf{B}_i}\hat{V} -
	{}^{\mathbf{B}_i}V)
\end{equation}
which can be made negative by choosing  $\mathbf{L}_{\mathbf{B}_i}
> M_{c,i}M_{v,i}I_{6\times 6}$. We will fix the choice for
$\mathbf{L}_{\mathbf{B}_i}$ later when designing the combined
controller-observer in Section \ref{sec:ctrl}.

\subsection{Observer for Joint $i$} \label{obs:joint}

Similarly as in the case of link $i$, we design a velocity observer for an arbitrary joint $i$, namely
\begin{subequations}
	\label{eq:obsJ1}% 
	\begin{align}
		\dot{\hat{q}}_i & = z_i - L_i(\hat{q}_i - q_i) \\
		I_m\dot{z}_i & = \tau_i - \tau_{ai} - f_{c,i}(\dot{\hat{q}}_i) - \ell_i(\hat{q}_i-q_i) 
		\label{eq:obsJ1b} 
	\end{align} 
\end{subequations}
where $[\hat{q}_i \quad z_i]^T$ is the observer state, $L_i,\ell_i > 0$ are
gain parameters and $\dot{\hat{q}}_i$ is the observed (angular)
velocity. Unlike the observer for link $i$, the proposed observer also
contains a position error feedback term which is added to achieve
position convergence in addition to velocity convergence.

Before computing the error dynamics, we set $L_i = \ell_i +
I_{m,i}^{-1}$. Now, subtracting \eqref{eq:tau_1} from
\eqref{eq:obsJ1b}, we obtain error dynamics
\begin{align}
	I_{m,i}(\ddot{\hat{q}}_i - \ddot{q}_i) & =
	-[f_{c,i}(\dot{\hat{q}}_i) -
	f_{c,i}(\dot{q}_i)]
	- (I_{m,i}\ell_i + 1)(\dot{\hat{q}}_i - \dot{q}_i) - \ell_i(\hat{q}_i-q_i),
\end{align}
which by defining a new variable $s_i := (\dot{\hat{q}}_i - \dot{q}_i)
+ \ell_i(\hat{q}_i-q_i)$ can be equivalently written as
$
I_{m,i} \dot{s}_i =
-[f_{c,i}(\dot{\hat{q}}_i) -
f_{c,i}(\dot{q}_i)] - s_i.
$
Let us now define a quadratic function $\nu_{i,obs}$ as
\begin{equation}
	\label{eq:vobsJ1}
	\nu_{i,obs} := \frac{I_{m,i}}{2}(\dot{\hat{q}}_i - \dot{q}_i)^2 +
	\frac{\ell_i}{2}(\hat{q}_i - q_i)^2 + \frac{I_{m,i}}{2}s_i^2.
\end{equation}
Then, denoting the Lipschitz constant of $f_{c,i}$ by $m_{c,i}$, we obtain
\begin{align}
	\dot{\nu}_{i,obs} & = -(\dot{\hat{q}}_i - \dot{q}_i)(f_{c,i}(\dot{\hat{q}}_i)
	- f_{c,i}(\dot{q}_i)) - I_{m,i}L_i(\dot{\hat{q}}_i - \dot{q}_i)^2 \nonumber \\
	& \quad - \ell_i(\dot{\hat{q}}_i - \dot{q}_i)(\hat{q}_i-q_i) +
	\ell_i(\dot{\hat{q}}_i - \dot{q}_i)(\hat{q}_i-q_i) \nonumber \\
	& \quad -s_i^2 - s_i(f_{c,i}(\dot{\hat{q}}_i) - f_{c,i}(\dot{q}_i)) \nonumber \\
	& \leq - I_{m,i}L_i(\dot{\hat{q}}_i - \dot{q}_i)^2 - s_i^2 -
	s_i(f_{c,i}(\dot{\hat{q}}_i) - f_{c,i}(\dot{q}_i)) \nonumber \\
	& \leq - I_{m,i}L_i(\dot{\hat{q}}_i - \dot{q}_i)^2 - s_i^2 +
	\frac{s_i^2}{2} + \frac{1}{2}(f_{c,i}(\dot{\hat{q}}_i) - f_{c,i}(\dot{q}_i))^2
	\nonumber \\
	& \leq - I_{m,i}L_i(\dot{\hat{q}}_i - \dot{q}_i)^2 - s_i^2 +
	\frac{s_i^2}{2} + \frac{m^2_{c,i}}{2}(\dot{\hat{q}}_i-\dot{q}_i)^2 \nonumber \\
	& = - \left(I_{m,i}L_i - \frac{m^2_{c,i}}{2}\right)(\dot{\hat{q}}_i - \dot{q}_i)^2 - \frac{1}{2}s_i^2, 
	\label{eq:vobsJ1d}
\end{align}
the right-hand-side of which is negative for all $L_i >
\frac{1}{2}m^2_{c,i}I_{m,i}^{-1}$. We will fix the
choice of $L_i$ (or rather $\ell_i$) as part of the joint
controller-observer design in the next section.

\begin{remark} \label{rem:sprop}
	Note that $\ell_i(\hat{q}_i - q_i) = s_i -
	(\dot{\hat{q}}_i-\dot{q}_i)$ so that $0 \leq \ell_i(\hat{q}_i - q_i)^2
	\leq 2\ell_i^{-1}s_i^2 + 2\ell_i^{-1}(\dot{\hat{q}}_i -
	\dot{q}_i)^2$.
\end{remark}

\section{Control with Observers}
\label{sec:ctrl}

In order to achieve position control for the system, let us introduce the
concept of \textit{required joint $i$ velocity} as $\dot{q}_{ir} =
\dot{q}_{id} + \lambda_i(q_{id} - q_i)$,
where $q_{id}$ is the traditionally used desired position trajectory (see \cite{siciliano2010robotics} 
for $q_{id}$) for joint $i$ and
$\lambda_i > 0$ is a control parameter
\cite[Sect. 3.3.6]{ZhuBook}. However, because we later need to be able
to realize $\ddot{q}_{ir}$ for the control, we redefine the required
velocity according to \cite[Sect. III.A]{BerNij93} as
\begin{equation}
	\dot{q}_{ir} := \dot{q}_{id} + \lambda_i({q}_{id} -
	{\hat{q}}_{i}) \label{eq:reqq},
\end{equation}
where we use the observed position $\hat{q}_i$ in place of $q_i$.

\subsection{Control of Link $i$}
\label{sec:link_ctrl}

In line with \eqref{eq:B1V}, the required linear/angular velocity vector at link $i$ base frame 
$\{\mathbf{B}_i\}$ can be obtained as
\begin{equation}
	\label{eq:BiVr}
	{}^{\mathbf{B}_i}V_r = \mathbf{z}_{\tau}\dot{q}_{ir} +
	{}^{\mathbf{B}_{i-1}}\mathbf{U}_{\mathbf{B}_i}^T{}^{\mathbf{B}_{i-1}}V_r
	\quad i \in \left\{ 1,2,\ldots,n \right\}, 
\end{equation}
and, in line with \eqref{eq:T1V}, the following relation holds for transforming the required 
linear/angular velocity vectors in link $i$:
\begin{equation} {}^{\mathbf{T}_i}V_r =
	{}^{\mathbf{B}_i}{\mathbf{U}}^T_{{\mathbf{T}_i}}{}^{\mathbf{B}_i}V_r,
	\quad i \in \left\{ 1,2,\ldots,n \right\} \label{eq:T1Vr} 
\end{equation}
with $^{\mathbf{B}_0}V_r = \mathbf{0}$ in \eqref{eq:BiVr} and \eqref{eq:T1Vr}.
Then, in view of \eqref{eq:B1F*} and using \eqref{eq:BiVr}, the required net
force/moment vector for link $i$ can be written as
\begin{equation}
		{}^{{\mathbf B}_i}F^*_{r}
		= {\mathbf M}_{{\mathbf B}_i}\frac{d}{dt}({}^{{\mathbf
				B}_i}{V_{r}})+{\mathbf C}_{{\mathbf B}_i}({}^{{\mathbf
				B}_i}{\hat{\omega}}){}^{{\mathbf B}_i}{V_{r}}+{\mathbf G}_{{\mathbf
				B}_i}  + {\mathbf K}_{{\mathbf B}_i}({}^{{{\mathbf
					B}_i}}V_{r} - {}^{{{\mathbf B}_i}}\hat{V}), \quad i \in \left\{ 1,2,\ldots,n \right\} 
	\label{eq:B1F*r}
\end{equation}
where ${\mathbf K}_{{\mathbf B}_i}  > 0$ is a velocity gain
matrix. Note that we need to use the observed velocities in the matrix
$\mathbf{C}_{\mathbf{B}_i}$ and in the feedback term. 
Finally, the required force/mo- ment vector can be written by reusing
\eqref{eq:B1F} as
\begin{equation}
	{}^{\mathbf{B}_i}F_{r} = {}^{\mathbf{B}_i}F^*_{r} +
	{}^{\mathbf{B}_i}{\mathbf{U}}_{{\mathbf{T}_i}}{}^{\mathbf{T}_i}F_{
		r}, \quad i \in \{1,2,\ldots,n\}. \label{eq:B1Fr}
\end{equation}

Let us define a quadratic function $\nu_{{\mathbf B}_{i,ctrl}}$ for
link $i$ as
\begin{equation}
	\nu_{{\mathbf B}_{i,ctrl}} := \frac{1}{2}({}^{{\mathbf
			B}_{i}}V_{r} - {}^{{\mathbf B}_{i}}V)^T{{\mathbf
			M}_{{\mathbf B}_{i}}}({}^{{\mathbf B}_{i}}V_{r} -
	{}^{{\mathbf B}_{i}}V). \label{eq:nu_B1c}
\end{equation}
Motivated by the discussion in \cite[Sect. II.C]{BerNij93}, we define
a quadratic function $\nu_{\mathbf{B}_i}$ for link $i$ as
\begin{equation}
	\label{eq:nuB1t}
	\nu_{\mathbf{B}_i} := \nu_{\mathbf{B}_i,ctrl} + \nu_{\mathbf{B}_i,obs},
\end{equation}
where $\nu_{\mathbf{B}_1,obs}$ is given
in \eqref{eq:vb1obs}. The following lemma provides an auxiliary result
that will be utilized in Section \ref{sec:stability} when proving
exponential convergences of the observation and control for the whole
$n$-DoF system.
\begin{lemma} \label{lem:link1}
	If the controller gain $\mathbf{K}_{\mathbf{B}_i}$ in
	\eqref{eq:B1F*r} is chosen such that $\mathbf{K}_{\mathbf{B}_i} >
	I_{6\times 6}$ and the observer gain $\mathbf{L}_{\mathbf{B}_i}$ in \eqref{eq:obsl1}
	is chosen such that $\mathbf{L}_{\mathbf{B}_i} > M_{c,i}M_{v,i} \left(1 +
	\frac{1}{2}M_{c,i}M_{v,i} \right)I_{6\times 6} \\ +
	\frac{1}{2}\mathbf{K}_{\mathbf{B}_i}$, then there exist some
	$M_{i,1}, M_{i,2} > 0$ such that $\nu_{\mathbf{B}_i}$
	in \eqref{eq:nuB1t} satisfies
	\begin{equation}
		\begin{split}
			\dot{\nu}_{{\mathbf B}_{i}} & \leq -({}^{{\mathbf B}_{i}}V_r -
			{}^{{\mathbf B}_{i}}V)^T{M_{i,1}}({}^{{\mathbf B}_{i}}V_{r} -
			{}^{{\mathbf B}_i}V) \\
			& \quad -(^{\mathbf{B}_i}\hat{V} -
			{}^{\mathbf{B}_i}V)^TM_{i,2}(^{\mathbf{B}_i}\hat{V} - {}^{\mathbf{B}_i}V) 
			+ p_{{\mathbf B}_{i}} - p_{{\mathbf T}_{i}}  \label{eq:nu_B1_dot}
		\end{split}
	\end{equation}
	for all $i \in \{1,2,\ldots,n\}$. 
\end{lemma}
\begin{proof} See Appendix \ref{proof:link1}.
\end{proof}

\subsection{Control of Joint $i$}
\label{sec:joint_ctrl}

We already saw in the previous section that, in line with \eqref{eq:B1V}, the required joint and link 
velocities are related by \eqref{eq:BiVr}. Alternatively, in line with 
\eqref{eq:B1Valt} we can write
\begin{equation}
	\label{eq:B1Vralt}
	^{\mathbf{B}_i}V_r = \mathbf{z}_{\tau} \dot{q}_{ir} +
	{}^{\mathbf{T}_{i-1}}U_{\mathbf{B}_i}^T{}^{\mathbf{T}_{i-1}}V_r,
	\quad i \in \left\{ 2,3,\ldots,n \right\}. 
\end{equation}
Then, in view of \eqref{eq:tau_a1}--\eqref{eq:tau_1}, the control law
for joint $i$ can be written as 
\begin{subequations}
	\label{eq:jointc}%
	\begin{align}
		\tau_{air} &= \mathbf{z}^T_\tau{}^{\mathbf{B}_i}F_{
			r} \\
		\tau_i &= I_{m,i}\ddot{q}_{ir} +
		f_{c,i}(\dot{q}_{ir}) + \tau_{air} +
		k_i(\dot{q}_{ir} - \dot{\hat{q}}_{i}) \label{eq:jc2} \\
		^{\mathbf{B}_0}F_r & = {}^{\mathbf{B}_0}U_{\mathbf{B}_1} {}^{\mathbf{B}_1}F_r
	\end{align}
\end{subequations}
where $k_i > 0$ is a velocity feedback gain. Similar to Lemma
\ref{lem:link1}, we have the following auxiliary result:
\begin{lemma} \label{lem:joint1}
	For all $k_i > 0$ in  \eqref{eq:jointc}, if the observer gains
	$\ell_i$ and $L_i$ in \eqref{eq:obsJ1} are chosen such that $2I_{m,i}L_i >
	\max\{2, m_{c,i}^2 + k_i\}$ and $\ell_i = L_i -
	I_{m,i}^{-1} > 0$, then there exists some $m_i > 0$ such that the quadratic function
	\begin{equation}
		\nu_{ai} := \frac{I_{m,i}}{2}(\dot{q}_{ir} -
		\dot{q}_{i})^2 +
		\frac{I_{m,i}}{2}(\dot{\hat{q}}_i-\dot{q}_i)^2 +
		\frac{\ell_i}{2}(\hat{q}_i-q_i)^2 + \frac{I_{m,i}}{2}s_i^2,
		\label{eq:nu_a1}
	\end{equation} 
	where $s_i = (\dot{\hat{q}}_i + \dot{q}_i)
	+ \ell_i(\hat{q}_i-q_i)$, satisfies
	\begin{equation}
		\dot{\nu}_{ai} \leq -\frac{1}{2}k_i(\dot{q}_{ir} -
		\dot{q}_i)^2 -m_i(\dot{\hat{q}}_i-\dot{q}_i)^2  -\frac{1}{2}s_i^2  + p_{{\mathbf A}_{i-1}} - 
		p_{{\mathbf
				B}_{i}}
		\label{eq:nu_a1_dot}
	\end{equation}
	for all $i \in \{1,2,\ldots,n\}$,
	where we denote $p_{{\mathbf A}_{0}} = p_{{\mathbf B}_{0}}$ and
	$p_{{\mathbf A}_{i}} = p_{{\mathbf T}_{i}}$ for $i\in \{1,2,\ldots,\\n-1\}$.
\end{lemma}

\begin{proof} See Appendix \ref{proof:joint1}.
\end{proof}

\section{Stability of the Entire System} \label{sec:stability}

In order to prove that the proposed controller-observer design
achieves tracking of the desired trajectories, recall that the base 
frame $\left\{ \mathbf{B}_0 \right\}$ has zero velocity and that no
external forces\footnote{Constrained motion control (i.e., contacts
	with the environment) can be addressed in VDC with a VPF appearing
	between the manipulator and the environment (see
	\cite{ZhuBook,Koivumaki_TRO2015,Koivumaki_TMECH2016}), but this
	topic is outside the scope of the present study.} are imposed on the
origin of frame $\left\{ 
\mathbf{T}_n \right\}$. Thus, $^{\mathbf{B}_0}V =
{}^{\mathbf{B}_0}V_{r} = \mathbf{0}$ and $^{\mathbf{T}_n}F =
{}^{\mathbf{T}_n}F_r = \mathbf{0}$ so that $p_{{\mathbf B}_{0}} = 
p_{{\mathbf T}_{n}} = 0$. 
Now we can construct a Lyapunov function for the
whole system by summing over the functions $\nu_{\mathbf{B}_i}$ and $\nu_{ai}$, as
the virtual power flows appearing in the time derivatives of the functions
will cancel out in the summation.

In Section \ref{obs:link} we assumed that the link velocities are bounded, but---as shown in the 
proof of Theorem \ref{thm:stability} below---this can be guaranteed by assuming the desired joint
velocities $\dot{q}_{id}$ to be bounded and by restricting to a
suitable set of initial conditions. Thus, we make the following
assumption:
\begin{assumption}
	There exist some $M_{d,i}, M'_{d,i} > 0$ such that $|q_{id}| \leq M_{d,i}$ and $|\dot q_{id}| \leq 
	M'_{d,i}$ for all $t\geq0$ and $i\in \left\{1,2,\ldots,n \right\}$.
	\label{ass:qdbdd}
\end{assumption} 
Moreover, for simplicity we assume that the gain matrices
$\mathbf{L}_{\mathbf{B}_i}, \mathbf{K}_{\mathbf{B}_i}$ are constant
and diagonal, that is:
\begin{assumption} 
	$\mathbf{L}_{\mathbf{B}_i} = L_{\mathbf{B}_i}I_{6\times 6}$ and
	$\mathbf{K}_{\mathbf{B}_i} = K_{\mathbf{B}_i}I_{6\times 6}$ for all $i\in \left\{
	1,2,\ldots,n \right\}$, where  $L_{\mathbf{B}_i}, K_{\mathbf{B}_i} > 0$.
	\label{ass:kldiag}
\end{assumption}
We will now show that the combined observer-control law converges exponentially for all initial 
conditions satisfying
\begin{equation}
	\label{eq:roa}
	\|\mathbf{x}(0)\| < \min_{i\in\{1,2,\ldots,n\}} \left\{
	\sqrt{\frac{\alpha_m}{\alpha_M}}\frac{\frac{\sqrt{1+2L_{\mathbf{B}_i}-K_{\mathbf{B}_i}}-1}{M_{c,i}} 
	\displaystyle
		- \sum_{k=1}^iM_U^kM'_{d,k}}{1 +
		{\displaystyle\sum_{k=1}^iM_U^k}\frac{16\lambda_k}{4\lambda_k - \alpha_M^{-1}}} \right\}
\end{equation}
with $\hat{q}_i(0) = q_{id}(0)$ and $4\lambda_i>\alpha_M^{-1}$ for all
$i\in\{1,2,\ldots,n\}$, where 
\begin{align*}
	\mathbf{x}^T & = \left[({}^{\mathbf{B}_i}V_r-{}^{\mathbf{B}_i}V)^T,
	({}^{\mathbf{B}_i}\hat{V}-{}^{\mathbf{B}_i}V)^T,
	\dot{q}_{ir}-\dot{q}_i, \dot{\hat{q}}_i-\dot{q}_i,s_i\right]_{i=1}^n \\
	\alpha_m & = \min \left\{ \min(\sigma(\mathbf{M}_{\mathbf{B}_i})), I_{m,i}
	\right\}_{i=1}^n \\
	\alpha_M & = \max \left\{\max(\sigma(\mathbf{M}_{\mathbf{B}_i})),
	I_{m,i}+\ell_i^{-1} \right\}_{i=1}^n, \\
	M_U & = \max \left\{ 1,
	\|^{\mathbf{B}_0}\mathbf{U}_{\mathbf{B}_1}\|,
	\|^{\mathbf{B}_1}\mathbf{U}_{\mathbf{B}_2}\|,\ldots,
	\|^{\mathbf{B}_n-1}\mathbf{U}_{\mathbf{B}_n}\|\right\}, 
\end{align*}
where $\sigma(\cdot)$ denotes the set of eigenvalues. Note that the gains 
$\mathbf{L}_{\mathbf{B}_i}, \mathbf{K}_{\mathbf{B}_i}$ and $\lambda_i$ can be assigned 
independently for all $i \in \{1,2,\ldots,n\}$, and that the region characterized by \eqref{eq:roa} can 
be made arbitrarily large by increasing the gains $\mathbf{L}_{\mathbf{B}_i}$ (while keeping the 
other parameters fixed). Thus, the region of attraction is \textit{semiglobal}.

\begin{theorem} \label{thm:stability}
	Under the standing assumptions and for all initial conditions satisfying \eqref{eq:roa}, the 
	combined observer-control law described in  \eqref{eq:obsl1}, \eqref{eq:obsJ1},
	\eqref{eq:reqq}--\eqref{eq:B1Fr} and \eqref{eq:B1Vralt}--\eqref{eq:jointc} with the gains chosen 
	according to
	Lemmas \ref{lem:link1} and \ref{lem:joint1}, the tracking errors
	$q_{id} - q_i$ decay exponentially to zero for all $i \in \left\{
	1,2,\ldots,n \right\}$. 
\end{theorem} 
\begin{proof} 
	By Lemmas \ref{lem:link1} and \ref{lem:joint1} and using 
	$p_{{\mathbf B}_{0}} = p_{{\mathbf T}_{n}} = 0$,  the quadratic function
	\begin{equation} \label{eq:nu_tot} 
		\nu := \sum_{i=1}^n (\nu_{\mathbf{B}_i} + \nu_{ai}),
	\end{equation} 
	with $\nu_{\mathbf{B}_i}$ and $\nu_{ai}$ given in
	\eqref{eq:nuB1t} and \eqref{eq:nu_a1}, respectively, satisfies 
	\begin{equation}
		\begin{split} 
			\dot{\nu}
			& = \sum_{i=1}^n (\dot{\nu}_{\mathbf{B}_{i}} +
			\dot{\nu}_{ai})  \\
			& \leq \sum_{i=1}^n \left[-({}^{{\mathbf B}_{i}}V_{
				r} - {}^{{\mathbf B}_{i}}V)^TM_{i,1}({}^{{\mathbf B}_{i}}V_{r} -
			{}^{{\mathbf B}_{i}}V) + p_{\mathbf{B}_i} \right.
			\\ 
			& \quad \qquad -(^{\mathbf{B}_i}\hat{V} -
			{}^{\mathbf{B}_i}V)^TM_{i,2}(^{\mathbf{B}_i}\hat{V} - {}^{\mathbf{B}_i}V)
			- p_{\mathbf{T}_i}    \\
			& \left. \quad \qquad -k_i(\dot{q}_{ir} - \dot{q}_{i})^2 -
			m_i(\dot{\hat{q}}_i - \dot{q}_i)^2 - \frac{1}{2}s_i^2  + p_{\mathbf{T}_i} -
			p_{\mathbf{B}_i} \right]   \\
			& = \sum_{i=1}^n \left[-({}^{{\mathbf B}_{i}}V_{
				r} - {}^{{\mathbf B}_{i}}V)^TM_{i,1}({}^{{\mathbf B}_{i}}V_{r} -
			{}^{{\mathbf B}_{i}}V)
		 -(^{\mathbf{B}_i}\hat{V} -
			{}^{\mathbf{B}_i}V)^TM_{i,2}(^{\mathbf{B}_i}\hat{V} - {}^{\mathbf{B}_i}V) \right.
			\\
			& \quad \qquad \left. -k_i(\dot{q}_{ir} - \dot{q}_{i})^2 -
			m_i(\dot{\hat{q}}_i - \dot{q}_i)^2 - \frac{1}{2}s_i^2 \right].
		\end{split}
		\label{eq:nu_tot_dot}
	\end{equation}
	Using Remark \ref{rem:sprop} and \cite[Thm. 4.10]{KhalBook}, we
	have that $\frac{\alpha_m}{2}\|\mathbf{x}\|^2 \leq \nu
	\leq\frac{\alpha_M}{2}\|\mathbf{x}\|^2$. Moreover, by
	\eqref{eq:nu_tot_dot} and Lemmas \ref{lem:link1} and \ref{lem:joint1}
	we have that $\dot{\nu} \leq -\alpha_p\|\mathbf{x}\|^2$ for $\alpha_p
	= \min \left\{\min(\sigma(M_{i,1})), \min(\sigma(M_{i,2})),k_i,m_i,\frac{1}{2}
	\right\}_{i=1}^n > 0$. However, positivity of $\alpha_p$ requires
	that especially the gain condition in Lemma \ref{lem:link1} holds,
	which under Assumption \ref{ass:kldiag} reduces to
	\begin{equation}
		\label{eq:gcL}
		L_{\mathbf{B}_i} >
		M_{c,i}\|{}^{\mathbf{B}_i}V\|(1 +
		\frac{1}{2}M_{c,i}\|^{\mathbf{B}_i}V\|) +
		\frac{1}{2}K_{\mathbf{B}_i}
	\end{equation}
	for all $i \in \left\{
	1,2,\ldots,n \right\}$ and $t\geq 0$. We will show that this is achieved for all initial conditions 
	satisfying \eqref{eq:roa}.
	
	Let $i\in\{1,2,\ldots,n\}$ be arbitrary. In order to estimate
	$\|{}^{\mathbf{B}_i}V\|$, we begin by 
	subtracting $\dot{\hat{q}}_i$ from both sides of
	\eqref{eq:reqq}. Rearranging terms yields linear dynamics
	$\dot{q}_{id} - \dot{\hat{q}}_i = - \lambda_i(q_{id} - {\hat{q}}_{i})
	+ \dot{q}_{ir} - \dot{\hat{q}}_i$, where by the above Lyapunov
	analysis $|\dot{q}_{ir} - \dot{\hat{q}}_i| \leq |\dot{q}_{ir} -
	\dot{q}_i| + |\dot{\hat{q}}_i - 
	\dot{q}_i| \leq 2
	\sqrt{\frac{\alpha_M}{\alpha_m}}\|\mathbf{x}(0)\|\exp(-\frac{\alpha_p}{2\alpha_M}t)$
	for all $t\geq 0$. Thus, by the variation of parameters formula we obtain
	\begin{equation}
		\label{eq:qidest}
		q_{id} - \hat{q}_i = (q_{id}(0)-\hat{q}_i(0))e^{-\lambda_it} +
		\int\limits_0^te^{-\lambda_i(t-s)}(\dot{q}_{ir} - \dot{\hat{q}}_i)ds
	\end{equation}
	and since by assumption $q_{id}(0) = \hat{q}_i(0)$, we can estimate
	$$
	\sup_{t\geq 0}|q_{id} - \hat{q}_i| \leq 
	\frac{4}{\lambda_i - \frac{\alpha_p}{2\alpha_M}}\sqrt{\frac{\alpha_M}{\alpha_m}}\|\mathbf{x}(0)\|
	$$
	where $\lambda_i > \frac{\alpha_p}{2\alpha_M}$ by assumption. Thus, by \eqref{eq:reqq}
	we obtain
	$$
	\sup_{t\geq 0}|\dot{q}_{ir}| \leq \sup_{t\geq 0}|\dot{q}_{id}| +
	\frac{4\lambda_i}{\lambda_i -
		\frac{\alpha_p}{2\alpha_M}}\sqrt{\frac{\alpha_M}{\alpha_m}}\|\mathbf{x}(0)\| 
	$$
	and consequently by \eqref{eq:BiVr} and Assumption \ref{ass:qdbdd} we obtain
	\begin{align*}
		\sup_{t\geq 0}\|{}^{\mathbf{B}_i}V_r\| 
		& \leq \sum_{k=1}^{i}M_U^k\sup_{t\geq
			0}|\dot{q}_{kr}| \\
		& \leq \sum_{k=1}^{i}M_U^kM'_{d,k}+ \sqrt{\frac{\alpha_M}{\alpha_m}}\|\mathbf{x}(0)
		\|\sum_{k=1}^{i}M_U^k\frac{16\lambda_k}{4\lambda_k -
			\alpha_M^{-1}},
	\end{align*}
	where we also used $\alpha_p \leq \frac{1}{2}$. Finally, we have that
	\begin{align*}
		\sup_{t\geq 0}\|{}^{\mathbf{B}_i}V\|
		& \leq \sup_{t\geq
			0}\|{}^{\mathbf{B}_i}V_r\| + \sup_{t\geq
			0}\|{}^{\mathbf{B}_i}V_r-{}^{\mathbf{B}_i}V\| \\
		& \leq  \sup_{t\geq
			0}\|{}^{\mathbf{B}_i}V_r\| +
		\sqrt{\frac{\alpha_M}{\alpha_m}}\|\mathbf{x}(0)\| \\
		& < \frac{\sqrt{1 + 2L_{\mathbf{B}_i} - K_{\mathbf{B}_i}} - 1}{M_{c,i}}
	\end{align*}
	for all $\mathbf{x}(0)$ satisfying \eqref{eq:roa}, i.e., \eqref{eq:gcL}
	holds.
	
	Now that we have shown that the error dynamics is exponentially stable
	with a given region of attraction, this implies by
	linearity that $\hat{q}_i - q_i = \ell_i^{-1}(s_i - (\dot{\hat{q}}_i -
	\dot{q}_i))$ decay exponentially to zero for all $i \in \left\{
	1,2,\ldots,n \right\}$. Finally, based on \eqref{eq:qidest} we have
	that $q_{id} - q_i  = (q_{id} -
	\hat{q}_i) +  (\hat{q}_i - q_i)$ decay exponentially to zero for all
	$i \in \left\{ 1,2,\ldots,n \right\}$, which concludes the proof. 
\end{proof}

\begin{remark}
	As $^{\mathbf{B}_i}\hat{V} - {}^{\mathbf{B}_i}V\to 0$ for all $i\in\{1,2,\ldots,n\}$ by Theorem 
	\ref{thm:stability}, it follows that $^{\mathbf{B}_i}\hat{P}$ in \eqref{eq:obsl1} converges up to a 
	constant from $^{\mathbf{B}_i}P$ and hence remains bounded by Theorem \ref{thm:stability} and 
	Assumption \ref{ass:qdbdd}.
\end{remark}

\begin{remark} \label{rem:appl}
	Note that as long as position and total torque data is available, the observers are
	in fact independent of the coordinate frames as there are no
	observer-based virtual power flows between neighboring frames. That
	is, as long as we can make the observers stable at the subsystem level, the
	proposed observer design could potentially be incorporated into more general VDC designs 
	\cite[Sect. 4]{ZhuBook} as well. Note also that the present design is not limited to planar joint 
	configuration as the joint orientations can be altered freely by changing the direction vector 
	$\mathbf{z}_\tau$.
\end{remark}

\section{Numerical Simulation of a 2-DoF 
	Robot} \label{sec:sim}

For an example, consider a robot as in Fig.
\ref{fig:system} with two links of length $l_1 = l_2 = 1$. Similar to
\cite[Sect. 2.1]{QuZDawBook}, both links are modeled as point masses
$m_1 = m_2 = 1$ at the distal ends so that the rotational inertia for
both links is $I_1 = I_2 = m_2l_2^2 = 1$.

In line with \eqref{eq:linkdyn}, the link dynamics are given by 
\begin{equation}
	{{\mathbf M}}_{{\mathbf
			B}_i}\frac{d}{dt}({}^{{\mathbf B}_i}V)+{{\mathbf C}}_{{\mathbf
			B}_i}({}^{{\mathbf B}_i}{\omega }){}^{{\mathbf B}_i}V+{{\mathbf
			G}}_{{\mathbf B}_i} = {}^{{\mathbf B}_i}F^*
\end{equation}
where \cite[Sect. 3.4]{ZhuBook} 
\begin{align}
	\mathbf{M}_{\mathbf{B}_i} & = 
	\begin{bmatrix}
		m_i & 0 & 0 \\ 0 & m_i & m_il_i \\ 0 & m_il_i & I_i +
		m_il_i^2
	\end{bmatrix} = 
	\begin{bmatrix}
		1 & 0 & 0 \\ 0 & 1 & 1 \\ 0 & 1 & 2
	\end{bmatrix}, \\
	\mathbf{C}_{\mathbf{B}_i}(\omega) & = 
	\begin{bmatrix}
		0 & -m_i & - m_il_i \\ m_i & 0 & 0 \\ m_il_i & 0 & 0
	\end{bmatrix}\omega = 
	\begin{bmatrix}
		0 & -1 & -1 \\ 1 & 0 & 0 \\ 1 & 0 & 0 
	\end{bmatrix}\omega, 
\end{align}
and $\mathbf{G}_{\mathbf{B}_i} = [0,0,m_ig_il_i\cos(q_i)]^T = [0,0,9.81\times\cos(q_i)]^T$ for $i \in 
\left\{ 1,2 \right\}$. The net forces $^{\mathbf{B}_i}
F^{*}$ are obtained based on \eqref{eq:B1F}, where the
transformation matrices are given in \cite[Sect 3.3.2]{ZhuBook}. The
link dynamics comprise two linear components and one angular
component. For further details on the 2-DoF example, see
\cite[Sect. 3]{ZhuBook}. The joint dynamics are given in line with \eqref{eq:tau_1} by $
I_{m,i}\ddot{q}_i = \tau_i - \tau_{ai} - f_{c,i}(\dot{q}_i)$,
where $I_{m,i} = 0.1$ and $f_{c,i}(\dot{q}_i) = \tanh(\dot{q}_i)$ for $i \in
\left\{ 1,2 \right\}$. Finally, $\tau_i$ is the input torque for
joint $i$ and $\tau_{ai}$ is given by \eqref{eq:tau_a1}.

For the simulation, the joint observer gains in \eqref{eq:obsJ1} are chosen as $\ell_1 =
200$ so that $L_i = 210$ for both joints, and the link observer gains
in \eqref{eq:obsl1} are chosen as 
$\mathbf{L}_{\mathbf{B}_i} = 200\times
I_{3\times 3}$ for both links. The link control gains in
\eqref{eq:B1Fr} are chosen as $\mathbf{K}_{\mathbf{B}_i} = 100\times I_{3\times 3}$ for
both links, and the joint control gains in \eqref{eq:jointc} are
chosen as $k_i  = 10$ for both joints. Moreover, for the required
velocities in \eqref{eq:reqq} the control parameter is chosen as
$\lambda_i = 10$ for both required velocities. The simulations are run
on a Simulink model corresponding to the dynamics presented in the
beginning of this section. 

The desired joint trajectories are given by $q_{1d}(t) = 0.8 - \cos
\left( \frac{\pi}{4}t \right)$ and $q_{2d}(t) = 0.8 - \cos \left(
\frac{\pi}{5}t\right)$.
The desired trajectories and the
joint position trajectories are displayed in Fig.
\ref{fig:pos}, where an initial error can be seen as the joints are
initially at $q_1(0) = q_2(0) = 0$ whereas the desired values are
$-0.2$. However, the initial error diminishes quickly, and thereafter
the joint trajectories follow the desired trajectories accurately.

\begin{figure}[!htbp]
	\centerline{\includegraphics[scale=0.56]{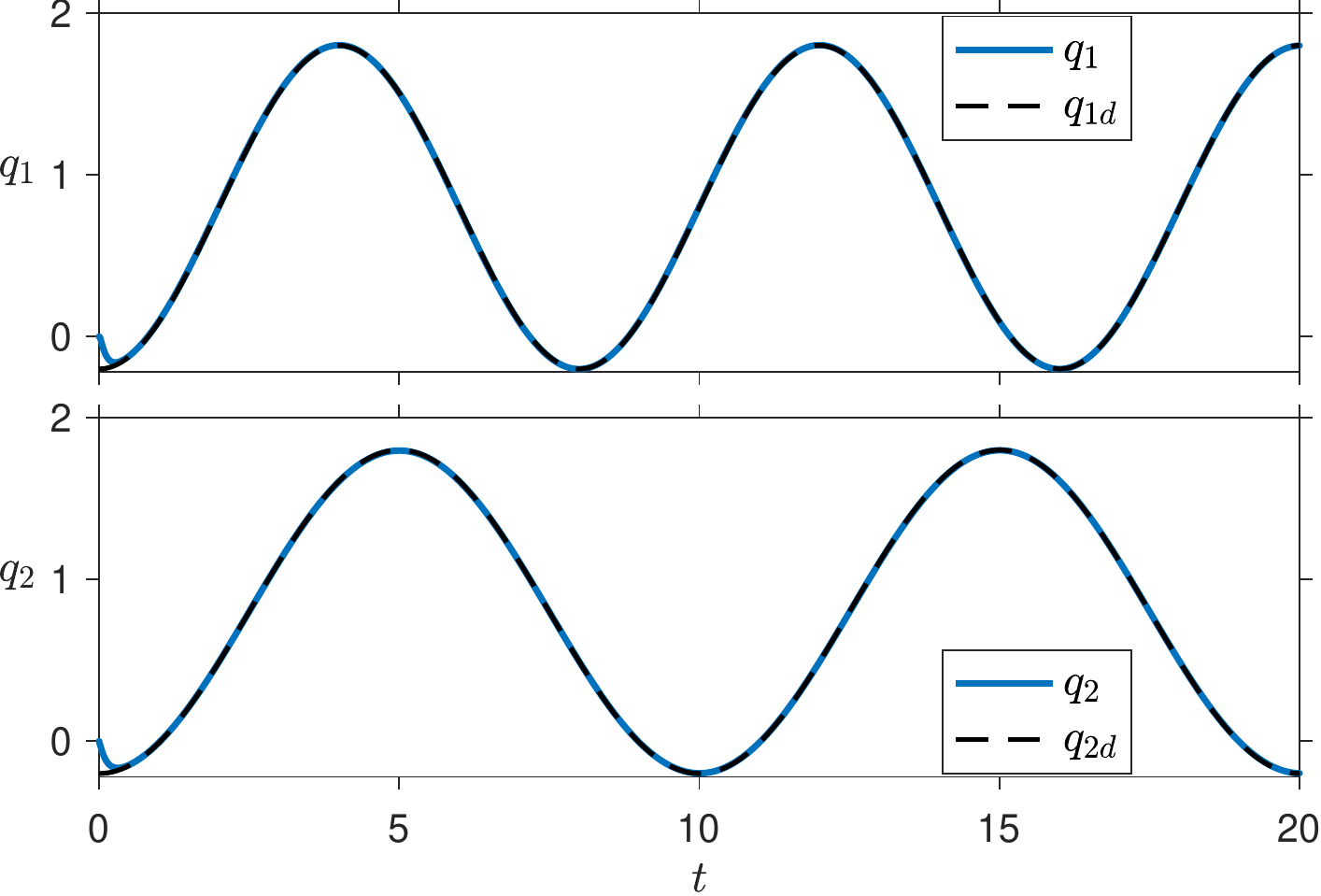}}
	\caption{Joint angle trajectories (in radians) and their desired values.}
	\label{fig:pos}
\end{figure}

Fig. \ref{fig:err} displays the tracking errors $e_i = q_i - q_{id}$ 
for both joints 1 and 2. The tracking errors
behave according to Fig. \ref{fig:pos}, that is, for both joints there is an initial
error of $0.2$ radians which diminishes rapidly, and thereafter the
position errors are virtually zero. The velocity observer errors
$\dot{\hat{q}}_i - \dot{q}_{id}$ are shown in Fig. \ref{fig:errv},
where one can see relatively large initial peaks as the initial
position tracking error is adjusted by the control input, but
thereafter the observed velocities are in accordance with the desired velocities.

\begin{figure}[htbp]
	\centerline{\includegraphics[scale=0.56]{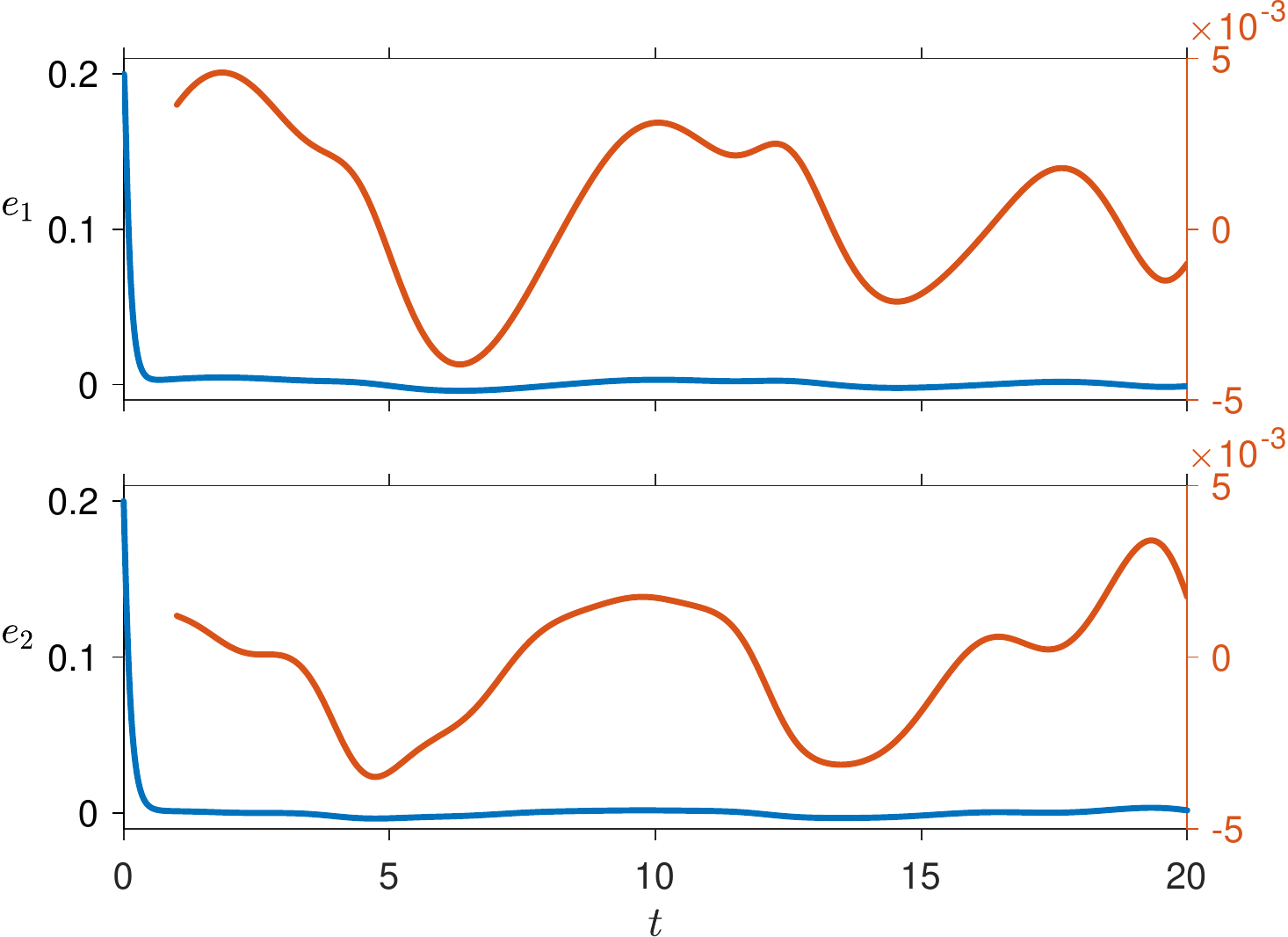}}
	\caption{Tracking errors $e_i = q_i - q_{id}$ (in radians) for joints
		$1$ and $2$. The asymptotic behavior for $t \geq 1$ is depicted in
		detail by the red lines and axes.}
	\label{fig:err}
\end{figure}

\begin{figure}[htbp]
	\centerline{\includegraphics[scale=0.56]{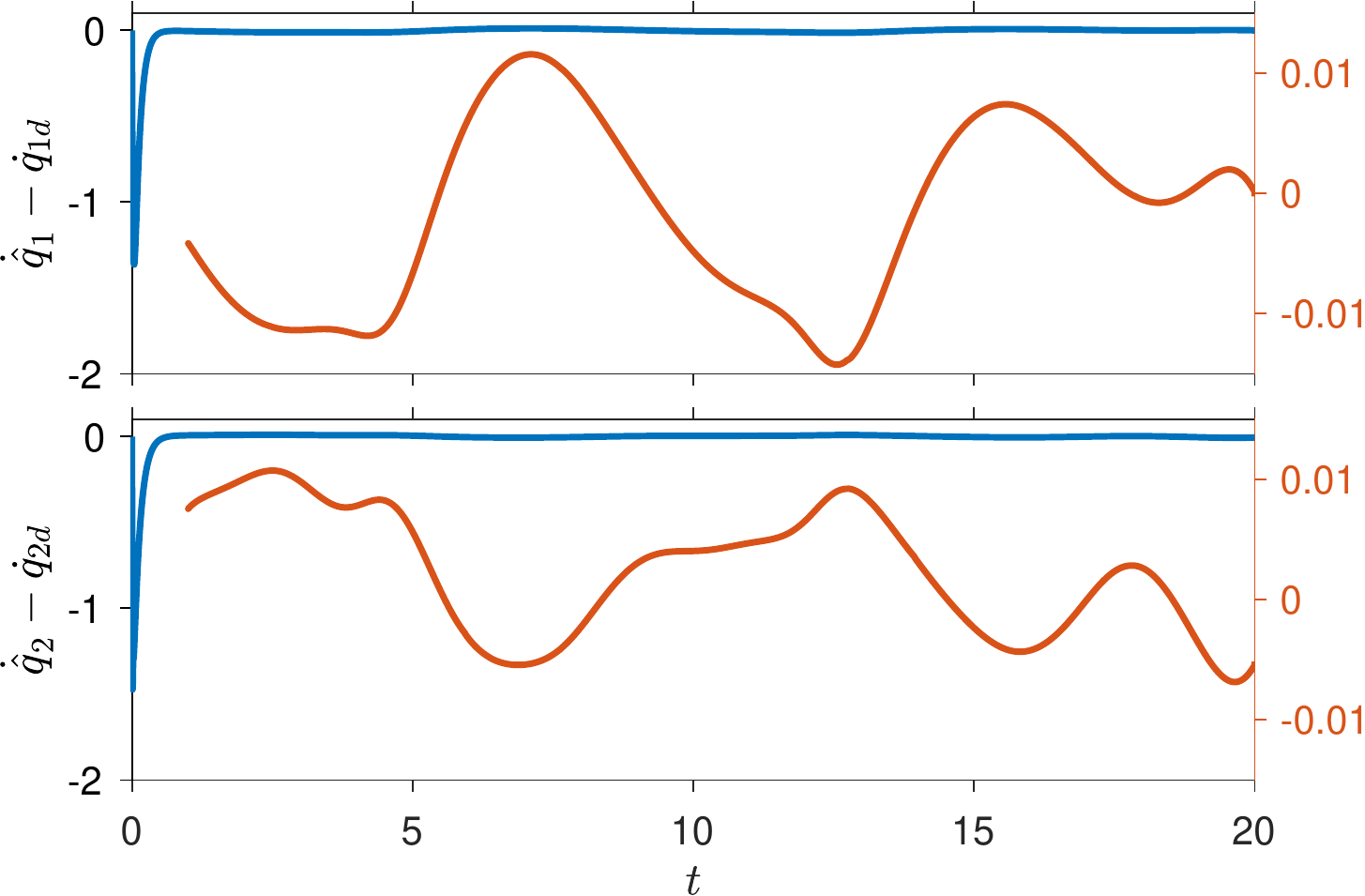}}
	\caption{Observer errors $\dot{\hat{q}}_i - \dot{q}_{id}$ (in
		radians/second) for joints $1$ and $2$. The asymptotic behavior for $t \geq 1$ is depicted in
		detail by the red lines and axes.}
	\label{fig:errv}
\end{figure}

\section{Conclusions} \label{sec:concl}

We incorporated a decentralized velocity observer design in the framework of virtual decomposition 
control of an open chain robotic manipulator. Stability analysis for the proposed controller-observer 
was carried out on a subsystem level by utilizing the concept of virtual stability. The observer error 
dynamics for a single subsystem were found to be independent of the other subsystems, which 
would suggest that the design could be extended to more complex systems as noted in Remark 
\ref{rem:appl}. In addition to proving the semiglobal exponential convergence of the combined 
controller-observer design, the proposed design was demonstrated in a simulation study of a 
$2$-DoF open chain system in the vertical plane. A topic for future research will be to incorporate 
parameter adaptation into the controller-observer design.

\appendix

\section{Proof for Lemma~\ref{lem:link1}}
\label{proof:link1}

First note that by subtracting \eqref{eq:B1F*} from \eqref{eq:B1F*r}, we obtain
\begin{equation}
	\begin{split}
		{}^{{\mathbf B}_{i}}F^{*}_{r} - {}^{{\mathbf
				B}_{i}}F^{*} &= {{\mathbf M}}_{{\mathbf B}_{i}}\frac{d}{dt}({}^{{\mathbf
				B}_{i}}V_{r} - {}^{{\mathbf B}_{i}}V) + {{\mathbf C}}_{{\mathbf
				B}_{i}}({}^{{\mathbf B}_{i}}{\hat{\omega} }){}^{{\mathbf B}_{i}}V_{r}
		\\ 
		& \quad -\mathbf{C}_{\mathbf{B}_i}(^{\mathbf{B}_i}\omega)
		{}^{{\mathbf B}_{i}}V + {{\mathbf
				K}}_{{\mathbf B}_{i}}({}^{{\mathbf B}_{i}}V_{r} - {}^{{\mathbf
				B}_{i}}\hat{V}). 
	\end{split}
	\label{eq:property1}
\end{equation}
Utilizing the properties of ${{\mathbf C}}_{{\mathbf B}_{i}}(\cdot)$
as in Section \ref{sec:obs} and using the
fact that $2V_1^TV_2 \leq \|V_1\|^2 +\|V_2\|^2$, we obtain
\begin{equation}
	\begin{split}
		& ({}^{{\mathbf B}_{i}}V_{r} - {}^{{\mathbf B}_{i}}V)^T[{{\mathbf C}}_{{\mathbf
				B}_{i}}({}^{{\mathbf B}_{i}}{\hat{\omega} }){}^{{\mathbf B}_{i}}V_{r} -
		{{\mathbf C}}_{{\mathbf
				B}_{i}}({}^{{\mathbf B}_{i}}{\omega }){}^{{\mathbf B}_{i}}V] \\
		& = ({}^{{\mathbf B}_{i}}V_{r} - {}^{{\mathbf B}_{i}}V)^T
		\mathbf{C}_{\mathbf{B}_i}(^{\mathbf{B}_i}\hat{\omega} -
		{}^{\mathbf{B}_i}\omega)^{\mathbf{B}_i}V \\
		& \leq \frac{1}{2}\|^{\mathbf{B}_i}V_r-{}^{\mathbf{B}_i}V\|^2
		+ \frac{1}{2}\|{{\mathbf C}}_{{\mathbf B}_{i}}(^{\mathbf{B}_i}\hat{\omega} -
		{}^{\mathbf{B}_i}\omega){}^{\mathbf{B}_i}V\|^2 \\
		& \leq \frac{1}{2}\|^{\mathbf{B}_i}V_r-{}^{\mathbf{B}_i}V\|^2
		+ \frac{1}{2}M_{c,i}^2\|^{\mathbf{B}_i}\hat{V}-{}^{\mathbf{B}_i}V\|^2M_{v,i}^2
	\end{split}
\end{equation}
where we also used the boundedness assumption of $^{\mathbf{B}_i}V$
and the relative boundedness \eqref{eq:Cp3} of $\mathbf{C}_{\mathbf{B}_i}(\cdot)$. Similarly, we 
obtain
\begin{align} 
	& -(^{\mathbf{B}_i}V_r - {}^{\mathbf{B}_i}V)^T
	\mathbf{K}_{\mathbf{B}_i}(^{\mathbf{B}_i}V_r -
	{}^{\mathbf{B}_i}\hat{V}) \nonumber \\
	& \leq -(^{\mathbf{B}_i}V_r - {}^{\mathbf{B}_i}V)^T
	\mathbf{K}_{\mathbf{B}_i} (^{\mathbf{B}_i}V_r - {}^{\mathbf{B}_i}V) \\
	& \qquad + \frac{1}{2}\|\sqrt{\mathbf{K}_{\mathbf{B}_i}}({}^{\mathbf{B}_i}V_r -
	{}^{\mathbf{B}_i}V)\|^2 +
	\frac{1}{2}\|\sqrt{\mathbf{K}_{\mathbf{B}_i}}({}^{\mathbf{B}_i}\hat{V} - 
	{}^{\mathbf{B}_i}V)\|^2. \nonumber
\end{align}

Moreover, using \eqref{eq:T1V}, \eqref{eq:B1F}, \eqref{eq:T1Vr} and
\eqref{eq:B1Fr} we obtain
\begin{align} 
		({}^{{\mathbf B}_{i}}V_{r} - {}^{{\mathbf
				B}_{i}}V)^T({}^{{\mathbf B}_{i}}F^*_{r} - {}^{{\mathbf
				B}_{i}}F^*) & = ({}^{{\mathbf
				B}_{i}}V_{r} - {}^{{\mathbf B}_{i}}V)^T\left[({}^{{\mathbf
				B}_{i}}F_{r} - {}^{{\mathbf B}_{i}}F) -
		{}^{\mathbf{B}_i}{\mathbf{U}}_{{\mathbf{T}_i}}({}^{\mathbf{T}_i}F_{
			r} - {}^{\mathbf{T}_i}F)\right]  \nonumber \\ & =
		p_{{\mathbf B}_{i}} -
		\left[{}^{\mathbf{B}_i}{\mathbf{U}}^T_{{\mathbf{T}_i}}({}^{{\mathbf
				B}_{i}}V_{r} - {}^{{\mathbf
				B}_{i}}V)\right]^T({}^{\mathbf{T}_i}F_{r} - {}^{\mathbf{T}_i}F)  \label{eq:vpfs}
		\\ &= p_{{\mathbf B}_{i}} - ({}^{{\mathbf
				T}_{i}}V_{r} - {}^{{\mathbf T}_{i}}V)^T({}^{\mathbf{T}_i}F_{r} - {}^{\mathbf{T}_i}F) = 
				p_{{\mathbf
				B}_{i}} - p_{{\mathbf T}_{i}}. \nonumber
\end{align}

Using \eqref{eq:property1}--\eqref{eq:vpfs} together with \eqref{eq:obsB1f}, ${\nu}_{{\mathbf 
		B}_{i}}$ in \eqref{eq:nuB1t} satisfies 
\begin{align*}
	\dot{\nu}_{{{\mathbf B}_i}} & =
	\dot{\nu}_{\mathbf{B}_i,ctrl} + \dot{\nu}_{\mathbf{B}_{i,obs}}
	\\ & =
	-({}^{{\mathbf B}_{i}}V_{r} - {}^{{\mathbf
			B}_{i}}V)^T[{{\mathbf C}}_{{\mathbf B}_{i}}({}^{{\mathbf
			B}_{i}}{\hat{\omega} }){}^{{\mathbf B}_{i}}V_{r} -
	\mathbf{C}_{\mathbf{B}_i}({}^{{\mathbf B}_{i}}\omega) {}^{{\mathbf
			B}_{i}}V] \\ & \quad - ({}^{{\mathbf
			B}_{i}}V_{r} - {}^{{\mathbf B}_{i}}V)^T{{\mathbf
			K}}_{{\mathbf B}_{i}}({}^{{\mathbf B}_{i}}V_{r} -
	{}^{{\mathbf B}_{i}}\hat{V}) \\ &\quad +
	({}^{{\mathbf B}_{i}}V_{r} - {}^{{\mathbf
			B}_{i}}V)^T({}^{{\mathbf B}_{i}}F^*_{r} - {}^{{\mathbf
			B}_{i}}F^*) + \dot{\nu}_{\mathbf{B}_{i,obs}}  \\
	&  \leq\frac{1}{2}\|^{\mathbf{B}_i}V_r-{}^{\mathbf{B}_i}V\|^2
	+\frac{1}{2}M_{c,i}^2\|^{\mathbf{B}_i}\hat{V}-{}^{\mathbf{B}_i}V\|^2M_{v,i}^2
	\\
	& \quad -(^{\mathbf{B}_i}V_r - {}^{\mathbf{B}_i}V)^T
	\mathbf{K}_{\mathbf{B}_i} (^{\mathbf{B}_i}V_r -
	{}^{\mathbf{B}_i}V) + p_{{\mathbf B}_{i}} - p_{{\mathbf
			T}_{i}}  \\
	& \qquad + \frac{1}{2}\|\sqrt{\mathbf{K}_{\mathbf{B}_i}}({}^{\mathbf{B}_i}V_r -
	{}^{\mathbf{B}_i}V)\|^2 +
	\frac{1}{2}\|\sqrt{\mathbf{K}_{\mathbf{B}_i}}({}^{\mathbf{B}_i}\hat{V} - 
	{}^{\mathbf{B}_i}V)\|^2 \\
	& \quad - (^{\mathbf{B}_i}\hat{V}-{}^{\mathbf{B}_{i}}V)^T(\mathbf{L}_{\mathbf{B}_i} -
	M_{c,i}M_{v,i}I_{6\times 6})(^{\mathbf{B}_i}\hat{V}-{}^{\mathbf{B}_{i}}V) 
	\\
	& = -(^{\mathbf{B}_i}V_r - {}^{\mathbf{B}_i}V)^T
	\left(\frac{1}{2} \mathbf{K}_{\mathbf{B}_i} -
	\frac{1}{2}I_{6\times 6} \right)(^{\mathbf{B}_i}V_r -
	{}^{\mathbf{B}_i}V) \\
	& \quad -(^{\mathbf{B}_i}\hat{V} -{}^{\mathbf{B}_i}V)^T \\ & \qquad 
	\times \left[\mathbf{L}_{\mathbf{B}_i}- M_{c,i}M_{v,i} \left(1 +
	\frac{1}{2}M_{c,i}M_{v,i}\right)I_{6\times 6} -
	\frac{1}{2} \mathbf{K}_{\mathbf{B}_i}  \right] \\
	& \qquad \times (^{\mathbf{B}_i}\hat{V} - {}^{\mathbf{B}_i}V) + p_{{\mathbf B}_{i}} - p_{{\mathbf
			T}_{i}},
\end{align*}
and the claim follows.

\section{Proof for Lemma~\ref{lem:joint1}}
\label{proof:joint1}

First note that by subtracting \eqref{eq:tau_1} from \eqref{eq:jc2}, we obtain
\begin{equation}
\tau_{air} - \tau_{ai} = -I_{m,i}(\ddot{q}_{ir}
		- \ddot{q}_{i}) - [f_{c,i}(\dot{q}_{ir}) - f_{c,i}(\dot{q}_{i})]  - k_{q,i}(\dot{q}_{ir} -
		\dot{\hat{q}}_{i}). 
	\label{eq:appB_property1}
\end{equation}
Using \eqref{eq:VPF}, \eqref{eq:B1V}, \eqref{eq:tau_a1}, \eqref{eq:B0F},
\eqref{eq:BiVr}, and \eqref{eq:jointc} we obtain for $i = 1$ that 
\begin{equation}
	\begin{split}
		(\dot{q}_{1r} - \dot{q}_1)(\tau_{a1r} - \tau_{a1})  
		& = (\dot{q}_{1r} - \dot{q}_1)
		\mathbf{z}_{\tau}^T(^{\mathbf{B}_1}F_r - {}^{\mathbf{B}_1}F)
		\\
		& = [^{\mathbf{B}_1}V_r - {}^{\mathbf{B}_1}V -
		{}^{\mathbf{B}_0}U_{\mathbf{B}_1}^T(^{\mathbf{B}_0}V_r -
		{}^{\mathbf{B}_0}V)]^T (^{\mathbf{B}_1}F_r - {}^{\mathbf{B}_1}F)
		\\
		& = p_{\mathbf{B}_1} - (^{\mathbf{B}_0}V_r -
		{}^{\mathbf{B}_0}V)^T
		{}^{\mathbf{B}_0}U_{\mathbf{B}_1}(^{\mathbf{B}_1}F_r -
		{}^{\mathbf{B}_1}F)   \\
		& = p_{\mathbf{B}_1} - (^{\mathbf{B}_0}V_r -
		{}^{\mathbf{B}_0}V)^T(^{\mathbf{B}_0}F_r -
		{}^{\mathbf{B}_0}F)  = p_{\mathbf{B}_1} - p_{\mathbf{B}_0}.
	\end{split}
	\label{eq:property21}
\end{equation}
Similarly, using \eqref{eq:VPF}, \eqref{eq:B1Valt}, \eqref{eq:tau_a1}, 
\eqref{eq:B1Vralt} and \eqref{eq:jointc}, we obtain for $i \in \left\{
2,3,\ldots,n \right\}$ that
\begin{align} 
	(\dot{q}_{ir} - \dot{q}_{i})(\tau_{air} -
	\tau_{ai}) \nonumber
	& = (\dot{q}_{ir} - \dot{q}_{i})
	\mathbf{z}^T_\tau({}^{\mathbf{B}_i}F_{
		r} - {}^{\mathbf{B}_i}F) \nonumber \\
	& = \left[({}^{\mathbf{B}_i}V_{r} -
	{}^{\mathbf{B}_i}V) -
	{}^{\mathbf{T}_{i-1}}{\mathbf{U}}^T_{{\mathbf{B}_i}}({}^{\mathbf{T}_{i-1}}V_{
		r} - {}^{\mathbf{T}_{i-1}}V)\right]^T({}^{\mathbf{B}_i}F_{
		r} - {}^{\mathbf{B}_i}F) \nonumber \\
	& = p_{\mathbf{B}_i} - ({}^{\mathbf{T}_{i-1}}V_{
		r} - {}^{\mathbf{T}_{i-1}}V)^T
	{}^{\mathbf{T}_{i-1}}{\mathbf{U}}_{{\mathbf{B}_i}} ({}^{\mathbf{B}_i}F_{
		r} - {}^{\mathbf{B}_i}F)  \label{eq:appB_property2}  \\
	& = p_{\mathbf{B}_i} - ({}^{\mathbf{T}_{i-1}}V_{
		r} - {}^{\mathbf{T}_{i-1}}V)^T({}^{\mathbf{T}_{i-1}}F_{
		r} - {}^{\mathbf{T}_{i-1}}F)
	= p_{\mathbf{B}_i} - p_{\mathbf{T}_{i-1}}.  \nonumber  
\end{align}
Using \eqref{eq:appB_property1}--\eqref{eq:appB_property2} together
with \eqref{eq:vobsJ1d},
$\nu_{ai}$ in \eqref{eq:nu_a1} satisfies
\begin{align*} 
	\dot{\nu}_{ai}  & =
	-k_i(\dot{q}_{ir} - \dot{q}_{i})^2 - [f_{c,i}(\dot{q}_{ir}) - f_{c,i}(\dot{q}_{
		i})](\dot{q}_{ir} - \dot{q}_{i})    \\
	&\quad  - (\dot{q}_{ir} - \dot{q}_{i})(\tau_{air} - \tau_{ai}) +
	k_i(\dot{q}_{ir}-\dot{q}_i)(\dot{\hat{q}}_i - \dot{q}_i) +
	\dot{\nu}_{i,obs}   \\
	& \leq -\frac{1}{2}k_i(\dot{q}_{ir}-\dot{q}_i)^2 
	- \left( I_{m,i}L_i - \frac{m^2_{c,i} + k_i}{2}
	\right)(\dot{\hat{q}}_i - \dot{q}_i)^2   \\
	& \quad   - \frac{1}{2}s_i^2 +p_{{\mathbf T}_{i-1}} - p_{{\mathbf B}_{i}},
\end{align*}
and the claim follows. 

\section*{Acknowledgements} 
The authors wish to thank the anonymous reviewers for their insightful comments and constructive 
suggestions.


\begin{thebibliography}{10}

\bibitem{ZhuBook}
W.-H. Zhu, \emph{Virtual Decomposition Control}.\hskip 1em plus 0.5em minus
0.4em\relax Springer, 2010.

\bibitem{Zhu1997}
W.-H. Zhu, Y.-G. Xi, Z.-J. Zhang, Z.~Bien, and J.~De~Schutter, ``Virtual
decomposition based control for generalized high dimensional robotic systems
with complicated structure,'' \emph{IEEE Trans. Robot. Autom.}, vol.~13,
no.~3, pp. 411--436, 1997.

\bibitem{Zhu1998adaptive}
W.-H. Zhu, Z.~Bien, and J.~De~Schutter, ``Adaptive motion/force~control of
multiple manipulators with joint flexibility based on virtual
decomposition,'' \emph{IEEE Trans. Autom. Control}, vol.~43, no.~1, pp.
46--60, 1998.

\bibitem{Zhu2000}
W.-H. Zhu and S.~E. Salcudean, ``Stability guaranteed teleoperation: an
adaptive motion/force control approach,'' \emph{IEEE Trans. Autom. Control},
vol.~45, no.~11, pp. 1951--1969, 2000.

\bibitem{Zhu2013}
W.-H. Zhu \emph{et~al.}, ``Precision control of modular robot manipulators: The
{VDC} approach with embedded {FPGA},'' \emph{IEEE Trans. on Robotics},
vol.~29, no.~5, pp. 1162--1179, 2013.

\bibitem{Koivumaki_TRO2015}
J.~Koivum{\"a}ki and J.~Mattila, ``Stability-guaranteed force-sensorless
contact force/motion control of heavy-duty hydraulic manipulators,''
\emph{IEEE Trans. Robot.}, vol.~31, no.~4, pp. 918--935, 2015.

\bibitem{Koivumaki_CEP2019}
J.~Koivum{\"a}ki, W.-H. Zhu, and J.~Mattila, ``Energy-efficient and
high-precision control of hydraulic robots,'' \emph{Control Engineering
	Practice}, vol.~85, pp. 176--193, 2019.

\bibitem{BerNij93}
H.~Berghuis and H.~Nijmeijer, ``A passivity approach to controller-observer
design for robots,'' \emph{IEEE Trans. Robot. Autom.}, vol.~9, no.~6, pp.
740--754, 1993.

\bibitem{Ber18arxiv}
S.~Berkane, ``A survey on output feedback control of robot manipulators with an
application to {PHANT}o{M} 1.5{A} haptic device,'' 2018, arXiv:1812.06809.

\bibitem{NicTom90}
S.~Nicosia and P.~Tomei, ``Robot control by using only joint position
measurements,'' \emph{IEEE Trans. Automat. Control}, vol.~35, no.~9, pp.
1058--1061, 1990.

\bibitem{ZhuChe92}
W.-H. Zhu, H.-T. Chen, and Z.-J. Zhang, ``A variable structure robot control
algorithm with an observer,'' \emph{IEEE Trans. Robot. Autom.}, vol.~8,
no.~4, pp. 486--492, 1992.

\bibitem{BurDav97}
T.~Burg, D.~Dawson, and P.~Vedagarbha, ``A redesigned {DCAL} controller without
velocity measurements: theory and demonstration,'' \emph{Robotica}, vol.~15,
pp. 337--346, 1997.

\bibitem{ZerDix99}
E.~Zergeroglu, W.~Dixon, D.~Haste, and D.~Dawson, ``A composite adaptive output
feedback tracking controller for robotic manipulators,'' \emph{Robotica},
vol.~17, pp. 591--600, 1999.

\bibitem{MalDri12}
S.~Malagari and B.~J. Driessen, ``Globally exponential controller/observer for
tracking in robots without velocity measurement,'' \emph{Asian J. Control},
vol.~14, no.~2, pp. 309--319, 2012.

\bibitem{Dri15}
B.~J. Driessen, ``Observer/controller with global practical stability for
tracking in robots without velocity measurement,'' \emph{Asian J. Control},
vol.~17, no.~5, pp. 1898--1913, 2015.

\bibitem{BuFYaoICRA00}
{Fanping Bu} and {Bin Yao}, ``Observer based coordinated adaptive robust
control of robot manipulators driven by single-rod hydraulic actuators,'' in
\emph{IEEE International Conference on Robotics and Automation}, vol.~3,
2000, pp. 3034--3039.

\bibitem{SirSal01}
M.~R. Sirouspour and S.~E. Salcudean, ``Nonlinear control of hydraulic
robots,'' \emph{IEEE Trans. Robot. Autom.}, vol.~17, no.~2, pp. 173--182,
2001.

\bibitem{KhalBook}
H.~K. Khalil, \emph{Nonlinear systems}, 3rd~ed.\hskip 1em plus 0.5em minus
0.4em\relax Upper Saddle River, NJ: Prentice-Hall, 2002.

\bibitem{AndSod07}
S.~Andersson, A.~S{\"o}derberg, and S.~Bj{\"o}rklund, ``Friction models for
sliding dry, boundary and mixed lubricated contacts,'' \emph{Tribology
	International}, vol.~40, pp. 580--587, 2007.

\bibitem{siciliano2010robotics}
B.~Siciliano, L.~Sciavicco, L.~Villani, and G.~Oriolo, \emph{Robotics:
	modelling, planning and control}.\hskip 1em plus 0.5em minus 0.4em\relax
Springer Science \& Business Media, 2010.

\bibitem{Koivumaki_TMECH2016}
J.~Koivum{\"a}ki and J.~Mattila, ``Stability-guaranteed impedance control of
hydraulic robotic manipulators,'' \emph{IEEE/ASME Trans. Mechatronics},
vol.~22, no.~2, pp. 601--612, 2016.

\bibitem{QuZDawBook}
Z.~Qu and D.~M. Dawson, \emph{Robust Tracking Control of Robot Manipulators},
1st~ed.\hskip 1em plus 0.5em minus 0.4em\relax IEEE Press, 1995.


\end{thebibliography}
\end{document}